\documentclass[11pt]{amsart}

\usepackage{amsmath,amssymb,amsthm,mathtools}
\usepackage{booktabs}
\usepackage[margin=1.2in]{geometry}
\usepackage{enumitem}
\usepackage[colorlinks=true, linkcolor=blue, citecolor=blue, urlcolor=blue]{hyperref}

\newtheorem{theorem}{Theorem}[section]
\newtheorem{lemma}[theorem]{Lemma}
\newtheorem{proposition}[theorem]{Proposition}
\newtheorem{corollary}[theorem]{Corollary}

\theoremstyle{definition}
\newtheorem{definition}[theorem]{Definition}
\newtheorem{example}[theorem]{Example}
\newtheorem{observation}[theorem]{Observation}

\theoremstyle{remark}
\newtheorem{remark}[theorem]{Remark}

\numberwithin{equation}{section}

\newcommand{\dd}{\delta}
\newcommand{\N}{\mathbb{N}}
\newcommand{\Z}{\mathbb{Z}}

\newcommand{\PP}{\mathbb{P}}
\newcommand{\eps}{\varepsilon}
\newcommand{\law}{\mathcal{L}}
\newcommand{\Sub}{\operatorname{Sub}}

\newcommand{\rev}[1]{\overleftarrow{#1}}
\newcommand{\mir}[1]{\overline{#1}}
\newcommand{\lz}{\ell_{0}}
\newcommand{\lo}{\ell_{1}}
\newcommand{\sub}{\trianglelefteq}
\newcommand{\nsub}{\ntrianglelefteq}
\DeclareMathOperator{\dens}{dens}
\newcommand{\mclE}{\mathcal{E}}

\begin{document}

\title[ The equality cases of the median property for $P_t$]{The equality cases $P_t(\mathbb{N})=\tfrac12$ for the deconvolved  sum-of-digits measures}

\author{Dawid Tar{\l}owski}
\address{Instytut Matematyki, Wydzia{\l} Matematyki i Informatyki, Uniwersytet Jagiello\'nski, ul.\ {\L}ojasiewicza 6, 30-348 Krak\'ow, Poland}
 \email{dawid.tarlowski@uj.edu.pl ; dawid.tarlowski@gmail.com}


\begin{abstract}
Let $s(n)$ denote the number of ones in the binary expansion of an integer $n\in\N$, and let $\mu_t$ be the probability measure on $\Z$ defined by the asymptotic densities of the level sets of the function $\N\ni n\mapsto s(n+t)-s(n)\in\mathbb{Z}$. Let $P_t$ be the family of finitely supported measures defined by the convolution $\mu_t=\mu_1*P_t$. Recently, Tarlowski (2026) has shown that the family $P_t$ may be represented as a recursively grown binary tree $T_t$, and that the Cusick's conjecture - $\mu_t(\N)>\frac12$, $t\in\N$, - follows from the asymmetry property of the family $T_t$, which was posed there as an open
problem. Next, Cheng (2026) has provided the combinatorial description of the family $T_t$ in the language of principal subsequence ideals, and proved both conjectures. Both of these problems are directly related to the problem of determining
the zeros of the function $\N\ni t \mapsto P_t(\N)-\frac12\in[0,\tfrac12]$,
a problem left open by Cheng (2026) as a saturation problem, and previously analyzed only numerically. In this paper we solve this problem completely. Writing an odd integer $t\ge3$ as $t=(1\,w\,1)_2$
with $w\in\{0,1\}^{\star}$, we show that $P_t(\N)=\frac12$ if and only if $w$ is \emph{saturated} in the following sense: in the block decomposition
$w=1^{a_0}\,0\,1^{a_1}\,0\cdots0\,1^{a_k}$ with exactly $k$ zeros, every block of "1"
satisfies $a_i\ge k$. Additionally, we show that the lower bound for $P_t(\N)$ established by Cheng for $0$-initial words holds true for all non-saturated words.

% , leaving open the problem of characterizing the equality cases $P_t(\N)=\frac12$. 

\end{abstract}

\maketitle

\section{Introduction}\label{sec:intro}

Let $s(n)$ denote the number of ones in the binary expansion of a natural number
$n\in\N=\{0,1,2,\dots\}$. For $t\in\N$ and $d\in\Z$, let
\begin{equation}\label{eq:mut}
\mu_t(d)\;=\;\dens\{n\in\N:\; s(n+t)-s(n)=d\}
\end{equation}
denote the asymptotic density of the corresponding level set. The densities
exist \cite{Besineau}, each $\mu_t$ is a probability measure on $\Z$ with mean
zero, and the family satisfies the recurrences
\begin{equation}\label{eq:murec}
\mu_{2t}=\mu_t,\qquad
\mu_{2t+1}=\tfrac12\,\sigma_{-1}\mu_{t+1}+\tfrac12\,\sigma_{1}\mu_{t},
\end{equation}
where $\sigma_d$ denotes the shift by $d$. Paper \cite{Tarlowski} studies the measures $\mu_t$ by considering the family $P_t$ defined by the recurrence
\eqref{eq:murec} with the initial measure $P_1=\dd_0$.
Equivalently,
$$\mu_t=\mu_1\ast P_t.$$
 The family $P_t$ is represented therein by a family of recursively grown binary trees $T_{t}$, namely,
$P_t$ is the probability distribution of the stopped random walk defined by
$T_t$, where $T_1=\bullet$, $T_3=[\bullet,\bullet]$, and, given
$T=T_{(1t_1t_2\dots t_n1)_2}=[T^-,T^+]$,
$$T_{(1t_1t_2\dots t_n01)_2}=[T,T^+]\mbox{ and }T_{(1t_1t_2\dots t_n11)_2}=[T^-,T].$$

Paper \cite{Tarlowski} conjectures that once the tree begins
to grow, exactly one side of it becomes heavier than $\frac12$, and it remains
so during the whole growth process; in the language of binary digits: for
every $n\ge1$ and $t_2,\dots,t_n\in\{0,1\}$,
\begin{equation}\label{TT}P_{(10t_2\dots t_n1)_2}(\mathbb{N})>\frac12\mbox{ and }P_{(11t_2\dots t_n1)_2}(-\mathbb{N})>\frac12.\end{equation}
By Lemma 25 from \cite{Tarlowski}, the  Cusick's conjecture (\cite{DKS,EmmeHubert,Spiegelhofer,SW,SS,Cheng}) which states that $\mu_t(\N)>\tfrac12$ for every $t\in\N$, is a special case of the above tree asymmetry  - for
every $t_1,\dots,t_n\in\{0,1\}$ and every $k\ge \lo(t_1\dots t_n)+2$, where
$\lo(u)$ denotes the number of ones in the word $u$, we have
\begin{equation}\label{ll}\mu_{(1t_1\dots t_n1)_2}(\N)=P_{(1t_1\dots t_n10^k1)_2}(\N)=P_{(10^k1t_n\dots t_1 1)_2}(\mathbb{N}),
\end{equation}
which directly implies:
\begin{equation}\label{M1}
\bigl\{\mu_t(\N)\colon t\in\N\bigr\}\;\subset\; \bigl\{P_t(\N)\colon t\in T\bigr\},
\qquad T:=2\N+1,
\end{equation}
and
\begin{equation}\label{M2}
\mbox{ if }\ \{t\in T\colon t_2=0\}\subset \{t\colon P_t(\N)>\tfrac12\}\ \mbox{ then }\ \{t\colon \mu_t(\N)>\tfrac12\}=\N,
\end{equation}
where $t_2$ denotes the second digit of the binary representation
$t=(t_1t_2\dots t_n)_2$. Relations \eqref{ll} and \eqref{M2} show that both
conjectures are questions about the elements of the following set
\begin{equation}\label{VV}\{t\in\N\colon P_t(\N)>\tfrac12\}.
\end{equation}
%Motivated by \eqref{M1} and \eqref{M2}, paper \cite{Tarlowski} presents
%numerical experiments concerning the global minima of $V$. While the
%numerical studies confirmed Conjecture 24, the general pattern for finding
%the set of minimizers
%$$\{t\in T\colon V(t)=\tfrac12\}$$ was not apparent.

 Paper \cite{Cheng} has provided the combinatorial description of the
family of trees $\{T_t\colon t\in \N\}$: the internal nodes of the tree
$T_{(1w1)_2}$ are labelled by the scattered subwords of the reversed word
$\rev{w}$. This provides the explicit formula for the stopping time $\tau_w$ determined by the tree $T_{(1w1)_2}$. Let
$u\sub a$ denote that $u$ is a subword of $a$. The stopping time $\tau_w$ determined by the tree $T_{(1w1)_2}$
is defined by the following first-exit time:
$$\tau_w(\omega)\;=\;\inf\{n\ge1\colon\ \omega_1\dots\omega_n\nsub\rev{w}\},\ \ \ \omega=(\omega_1,\omega_2,\dots)\in \{0,1\}^{\N^+}.$$
Moreover, the above formula leads to an explicit description for the measures $P_t$. Given a random walk on a tree $T_w$, the  walk exits the set of internal nodes when it enters the set of leaves of the tree, see Corollary \ref{cor:explicitP} below for more details, and hence:

$$P_w\;=\;\sum_{v\in\Lambda(w)}\Bigl(\tfrac12\Bigr)^{\ell(v)}\dd_{b(v)},$$
where $\Lambda(w)$ denotes the set of the labeled leaves of the tree - a word $v=v_1\dots v_{\ell(v)}$ labels a leaf of $T_w$ iff
$v_1\dots v_{\ell(v)-1}\sub\rev{w}$ and $v\nsub\rev{w}$. Above, $\ell(v)$ is the
length of $v$, and $b(v)=\ell_1(v)-\ell_0(v)$ is the position of the walk at the end of the path $v$. Paper \cite{Cheng} proves the following median property
 
\begin{equation}\label{eq:median}
P_t(\N)\ \ge\ \tfrac12
\qquad\text{and}\qquad
P_t(-\N)\ \ge\ \tfrac12
\qquad (t\ge1),
\end{equation}
conjectured already in \cite{DKS}, and the tree asymmetry \eqref{TT}, which directly forces Cusick conjecture . From now on, a digit word
$w\in\{0,1\}^{\star}$ will be identified with the odd integer
$w=(1\,w\,1)_2$, and we write $P_w:=P_{(1w1)_2}$ (see
Section~\ref{sec:prelim}). Paper \cite{Cheng} proved that
\begin{equation}\label{eq:cheng55}
P_{w}(\N)\;\ge\;\frac12+2^{-2\min(\lz(w),\,\lo(w)+1)-1},
\qquad\text{whenever $w$ begins with the digit $0$}
\end{equation}
 together with the complementary bound
for words beginning with $1$ on the $-\N$ side. The problem of characterization of the set $\{t\in\N\colon P_t(\N)=\frac12\}$ (the saturation problem) was left open in \cite{Cheng};  in \cite{Tarlowski}, this problem was analyzed numerically and the formula for the elements of this set, based on finite number of cases, was far from obvious. The main goal of this paper is to provide the exact characterization of odd integers $t$ with $P_t(\N)=\frac12$. First, note that
that every $w\in\{0,1\}^{\star}$ with $k=\lz(w)$ zeros can be written
uniquely as
\begin{equation}\label{eq:profile}
w\;=\;1^{a_0}\,0\,1^{a_1}\,0\,\cdots\,0\,1^{a_k},
\qquad a_0,\dots,a_k\in\N,
\end{equation}
where $a_i=0$ is allowed; we call $(a_0,\dots,a_k)$ the \emph{block profile}
of $w$. For $k=0$ the decomposition \eqref{eq:profile}
reads $w=1^{a_0}$, including the empty word $\eps$ (for which $a_0=0$).
\begin{definition}\label{def:saturated}
A word $w\in\{0,1\}^{\star}$ with $k=\lz(w)$ and block profile
$(a_0,\dots,a_k)$ is called \emph{saturated} if
\[
a_i\;\ge\;k\qquad\text{for every }0\le i\le k.
\]
Equivalently: every zero of $w$ is isolated, and every maximal
block of ones of $w$ --- including the first and the last one --- has length at
least $\lz(w)$. In particular, every word without the digit $0$ (in particular $\eps$) is
saturated.
\end{definition}
% We write
%$-\N=\{0,-1,-2,\dots\}$, and $\mir{w}$ for the \emph{complement} of $w$, the word obtained from $w$
%by interchanging the digits $0$ and $1$.
 Our main result is the following. 
\begin{theorem}\label{main}
Let $w\in\{0,1\}^{\star}$ and $t=(1\,w\,1)_2$. Then
\[
P_t(\N)=\dfrac12 \quad\text{ if and only if }\quad w \text{ is saturated.}
\]
\end{theorem}
Additionally, our Corollary~\ref{cor:main-ii} extends the lower bound \eqref{eq:cheng55} to all non-saturated words, that is, if $w$ is non-saturated, then we have
\[
P_{w}(\N)\;\ge\;\frac12+2^{-2\min(\lz(w),\,\lo(w)+1)-1}
\]

In terms of the binary representation, Theorem~\ref{main}
reads as follows: writing $t=(1\,w\,1)_2$, the parameters with $P_t(\N)=\frac12$ are
exactly the odd $t$ whose binary expansion has $k\ge0$ zeros, all of them
isolated, with every interior maximal block of ones of length at least $k$ and
the two outer blocks of ones of length at least $k+1$ (the outer blocks of the
expansion of $t$ carry the two flanking ones of $(1\,w\,1)_2$ in addition to the
outer blocks of $w$).
The saturated family has a transparent enumerative structure.
\begin{corollary}\label{cor:count}
For $\ell\ge0$, the number $E(\ell)$ of saturated words of length $\ell$
equals
\[
E(\ell)\;=\;\sum_{k\ge0}\binom{\ell-k(k+1)}{k},
\]
with the convention $\binom{n}{k}=0$ for $n<k$. The summand indexed by $k$ is
nonzero if and only if $\ell\ge k(k+2)$; the unique shortest saturated word
with $k$ zeros is the periodic word $1^{k}(0\,1^{k})^{k}$ of length
$k(k+2)$. Additionally, it is rather easy to note that the asymptotic density of the set $$\{t\in\N\colon P_t(\N)=\frac12\}$$ equals zero.
\end{corollary}

\section{Notation and preliminaries}\label{sec:prelim}

\subsection{Odd integers as digit words}\label{subsec:order}

We write $\N=\{0,1,2,\dots\}$, and the binary representation of an integer $t\in\N$ will be denoted by $(t_1\dots t_n)_2$. 
Given a nonempty set $A$, the $A^\star$ will denote the set of all finite words with letters from $A$:
$$A^\star=\bigcup_{n\in\N}A^{n}\ (A^0:=\{\eps\},\ \eps -\mbox{ an empty word}).$$
Let
$$T=2\N+1.$$

 The following bijection identifies the set $\{0,1\}^\star$, with the  set  $T_{\{\geq 3\}}=\{t\in T\colon t\geq 3\}$:
$$h\colon \{0,1\}^\star\ni w  \mapsto (1w1)_2\in T_{\{\geq 3\}}.$$ 

The encoding used in  \cite{Tarlowski} is analogous: given the two mappings $L,R\colon T\to T$ defined by
\[
L(t)=2t-1,\qquad R(t)=2t+1 ,
\]
any $t\in T_{\{> 3\}}$ may be uniquely written as $t=(w_n\circ\dots\circ w_1)(3)$, which leads to the identification:
$$t=w_1\dots w_n\in\{L,R\}^\star \Leftrightarrow t=(w_n\circ\dots\circ w_1)(3),\ (\mbox{ and }3=\eps\in\{L,R\}^\star). $$
 By \cite[Observation 3]{Tarlowski}, the above representation is equivalent to the binary representation, namely, the binary representation of $t=w_1\dots w_n\in\{L,R\}^\star$ satisfies:
$$t=(1\beta(w_1)\dots \beta(w_n)1)_2,\mbox{ where } \beta(L)=0,\ \ \beta(R)=1.$$
In particular, given $w\in\{0,1\}^\star$, we get
$$(1w01)_2=2\cdot (1w1)_2-1\mbox{ and }(1w11)_2=2\cdot (1w1)_2+1.$$

In this paper, the set $T=2\N+1$ will be identified directly with $\{0,1\}^\star$.

\subsection{Words}\label{subsec:words}
 Given $u=u_1\dots u_n\in\{0,1\}^\star$, 
we will write $\ell(u)=n$ and 
$$ \lz(u)=\sum\limits_{i\leq n}1_{\{0\}}(u_i)\mbox{ and }\lo(u)=\sum\limits_{i\leq n}1_{\{1\}}(u_i),$$  so we have
$\ell(u)=\lz(u)+\lo(u)$. The \emph{balance} of $u$ is
\[
b(u)\;=\;\lo(u)-\lz(u)\in\Z .
\]
For $u=u_1u_2\dots u_n$ we write $\rev{u}=u_n\dots u_2u_1$ for the
\emph{reversal} of $u$, and $\mir{u}$ for the \emph{complement} of $u$ - the word obtained
by interchanging the digits $0$ and $1$ ($\mir{\eps}:=\eps$ and $\rev{\eps}:=\eps$ ).

\subsection{Measures}\label{subsec:measures}
Probability measures $P$ on $\Z$ are identified with the vectors $(P(d))_{d\in\Z}$. For $d\in\Z$,
the shift operator $\sigma_d$ acts by $(\sigma_d P)(k)=P(k-d)$. Following
\cite[Definition 4]{Tarlowski}, let
\[
\Phi(\mu,\nu)\;=\;\tfrac12\,\sigma_{-1}\mu+\tfrac12\,\sigma_{1}\nu .
\]

As in \cite{Tarlowski} ,we  will work with the family of measures $\{P_t\}_{t\in\N^+}$ given by:
\begin{equation}\label{eq:Pt}
P_1=\delta_0,\qquad P_{2t}=P_t,\qquad
P_{2t+1}=\Phi(P_{t+1},P_t)\quad(t\ge1),
\end{equation}
By \cite[Corollary 7]{Tarlowski},
\begin{equation}\label{eq:deconv}
\mu_t=\mu_1*P_t,\qquad (\
\mu_1=\sum_{j\ge0}(\frac12)^{j+1}\delta_{1-j}\ ).
\end{equation}
By $P_{2t}=P_t$, we will focus on the odd indexes $t\in T$, that is, we are interested in the family
$$\{P_{w}\}_{ w\in\{0,1\}^{\star}}, \mbox{ where }P_w=P_{(1w1)_2}.$$

\subsection{Trees}
Let $\mathcal{T}$ denote the set of finite full planar binary trees: the leaf
$\bullet$ belongs to $\mathcal{T}$, and if $T^-,T^+\in\mathcal{T}$ then so
does $[T^-,T^+]$.

Paper \cite{Tarlowski} expresses the family $\{P_t\}_{t\in T}$ as the family of trees $\{T_t\}_{t\in T}.$
First, $T_1=\bullet$, and $T_3=[T_3^-,T_3^+]=[\bullet,\bullet]$. Next, given $T_w=[T_w^-,T_w^+]$, we have the following growth recursion:
\begin{equation}\label{growth}
T_{w0}=[\,T_w,\;T_w^{+}\,],\qquad
T_{w1}=[\,T_w^{-},\;T_w\,],
\end{equation}
By Theorem 11 in \cite{Tarlowski}, a measure $P_t$ is the probability distribution of the stopped random walk determined by $T_t$; full details are given below.

\subsection{Simple random walk}\label{subsec:space}
Let $(\Omega,\mathcal F,\PP)$ be the following canonical probability space:
\begin{itemize}
\item $\Omega\;=\;\{0,1\}^{\N^+}$, the elementary events are denoted by $w=(w_1,w_2,\dots)\in\Omega$
\item $\mathcal F=\sigma\bigl([u]:u\in\{0,1\}^{\star}\bigr)$ is the $\sigma$- field of cylinder sets:
$$[\eps]=\Omega\mbox{ and } \ [u]=\bigl\{\omega\in\Omega:\ \omega_1\dots\omega_{\ell(u)}=u\bigr\}\mbox{ for }u\neq\eps.$$
\item $\PP$ is the infinite product of the fair-coin measures $\frac12\delta_0+\frac12\delta_1$:

\[
\PP\;=\;\bigotimes_{i\ge1}\Bigl(\tfrac12\,\dd_{0}+\tfrac12\,\dd_{1}\Bigr).
\]

\end{itemize}

 For $n\ge0$, the truncation map is denoted by
$\pi_n\colon\Omega\to\{0,1\}^{n}$, that is
$\pi_n(\omega)=\omega_1\dots\omega_n$ ($\pi_0(\omega)=\eps$), so for
$u\in\{0,1\}^{\star}$, the cylinder set $[u]$ satisfies
\[
[u]\;=\;\pi_{\ell(u)}^{-1}\bigl(\{u\}\bigr) \]

In particular, for $u\in\{0,1\}^{\star}$,
\begin{equation}\label{eq:cylmass}
\PP\bigl([u]\bigr)=2^{-\ell(u)},
\qquad
[u]=[u0]\sqcup[u1]\qquad \mbox{(the disjoint sum)}.
\end{equation}
We write $\mathcal F_n=\sigma(\pi_n)$, $n\in\N$, for the
natural filtration, and
\[
\theta\colon\Omega\to\Omega,\qquad
\theta(\omega)=(\omega_2,\omega_3,\dots)
\]
for the shift on $\Omega$, which preserves $\PP$. The simple random walk is defined by
\begin{equation}\label{eq:walk}
S_0=0,\qquad S_n=\xi_1+\dots+\xi_n\quad(n\ge1),
\end{equation}
where $\xi_k$ is determined by the $k$-th coordinate of $w$: $\xi_k=\;\chi(\omega_k),\  k\ge1,$ where
\begin{equation}\label{eq:steps}
\chi(a)=2a-1,\ \text{ so we have: } \chi(0)=-1,\ \chi(1)=+1 .
\end{equation}

Thus, the random variables $\xi_1,\xi_2,\dots$ are independent with:
$$\PP(\xi_k=-1)=\PP(\xi_k=+1)=\frac12,$$ and, for $ u\in\{0,1\}^{n}$,
\begin{equation}\label{eq:Sonu}
S_n=b(u)\ \text{ on }[u] .
\qquad .
\end{equation}

%Finite digit words therefore play two distinct roles below: as parameters and
%as labels of tree nodes on the one hand, and as truncations $\pi_n(\omega)$ of
%a sample point on the other. The letter $\omega$ is reserved for elements of
%$\Omega$. The coincidence of the two alphabets is deliberate: it is exactly
%what will identify the internal nodes of $T_w$ with the truncations of the
%paths that have not yet been stopped (Section~\ref{sec:nodes}).

\subsection{Trees and stopping times}\label{subsec:trees}
Every tree $T\in\mathcal{T}$ determines a
bounded stopping time $\tau_T\colon\Omega\to\N$ with respect to the
filtration $(\mathcal F_n)_{n\ge0}$. Given $T=\bullet$,
$$\tau_{\bullet}\equiv 0,$$
and, given $T=[T^0,T^1]$,
$$ \tau_{[T^0,T^1]}(\omega)\;=\;1+\tau_{T^{\omega_1}}\bigl(\theta\omega\bigr),\  \omega\in\Omega.$$

Above, we see that the first coordinate $\omega_1$ selects the subtree into which the
walk descends, and the remaining coordinates determine the further path. The $S_{\tau_T}$ will denote the random walk stopped at $\tau_T$, and the embedding map $\mclE(T)=\law(S_{\tau_T})$ assigns to each tree the law of the corresponding stopped random walk. We have the equality: $\mclE([T^-,T^+])=\Phi(\mclE(T^-),\mclE(T^+))$, which implies that  the family of trees $\{T_w\}_{w\in \{{0,1}\}^\star}$ given by equation \eqref{growth} represents the measures $\{P_w\}_w$, \cite[Theorem 11]{Tarlowski}. More specifically, any $P_w$
 %be the family of trees defined by \eqref{growth} - we put $T_1=T_{\bullet}$, and recall that we identify the indexes $\T_{\geq 3}$ with $\{0,1\}^\star$ by $(1w1)_2=w$. By
 is  the probability distribution of the random walk stopped at time $\tau_w$ determined by the tree $T_w$:
\begin{equation}\label{eq:PisE}
P_w\;=\;\mclE(T_w)\;=\;\law\bigl(S_{\tau_w}\bigr),
\qquad \tau_w:=\tau_{T_w},\qquad w\in\{0,1\}^\star .
\end{equation}

Given $w\in\{0,1\}^\star$, the stopping time $\tau_w$ returns the number of steps after which the random walk reaches a leaf ($\bullet$) of the tree $T_w$. Equivalently, $\tau_w$ is the step at which the walk leaves the set of internal nodes of the tree (an "exit time", with use of the language of \cite{Cheng}). By definition, the depth of the tree $T_w$ equals to $\ell(w)+1$, and hence
$$\tau_w\leq \ell(w)+1.$$

The encoding \eqref{eq:steps} is consistent with the growth recursion
\eqref{growth}: the left/right, edge of a tree corresponds to the
digit $0$/$1$, and to the step $-1$/$+1$. In consequence, for
$u\in\{0,1\}^{n}$ the cylinder $[u]$ is precisely the event that the first $n$ steps of the random walk follow the path  labelled by $u$. 

\section{Random walk on a tree and the running mass of $\mathbb{N}$}
%, and, by
%\eqref{eq:cylmass} and \eqref{eq:Sonu},
%$$\PP([u])=\frac{1}{2^{\ell(u)}}\quad\mbox{ and }\quad S_n=b(u)\ \mbox{ on }[u].$$

Given $w\in\{0,1\}^\star$, we will consider the stopped random walk $S^w$ given by:
$$S^w_n=S_{\min(\tau_w,n)}=\sum\limits_{k=1}^{\min(\tau_w,n)}\xi_k,\qquad
(\tau_w=\tau_{T_w}).$$
The above process is the random walk on the tree $T_w$ which moves to the right/left with probability $1/2$ while it occupies any internale node of the tree, and it stops when it reaches any leaf  of the tree. By $\tau_w\leq \ell(w)+1$, we  have:
$$P_w=\law( S^w_{\ell(w)+1}).$$

 The questions regarding the inequalities $P_t(\N)\geq\frac12$ and $P_t(\N)>\frac12$ are directly related to the median preserving property of the stopped process $\{S^w_n\}_{n\in\N}$. For any $w\in\{0,1\}^\star$, we have:
$$ \PP[S^w_0\geq 0]=1\mbox{ and } \PP[S^w_1\geq 0]=\frac12.$$ 
It is easy to express how the probability $\PP(S^w_n\in \N)$ changes from one step to another- any running path of the random walk splits the mass of every point as long as the walk does not leave the set of the internal nodes ( an internal node passes half of its weight to each child at one step). Define 
$$\eta_w=\tau_w-1$$
so the value of $\eta_w$ is the step at which the random walk $S^w_n$ hits the set of the internal nodes for the last time.\\

The control of 
$$P_t(\N)=\mathbb{P}( S^w_{\ell(w)+1}\in \N)$$
is about describing the passages of the process $S^w_n=S_{\min(\tau_w,n)}$ between the positions $-1$ and $0$.  For instance, for $n\in2\mathbb{N}$, any path $u=u_1\dots u_n\in\{0,1\}
^\star$ with $b(u)=0$ and $n\leq \eta_w(u)$ will decrease the value of the running mass of $\N$ by the value $\frac12\cdot \PP([u])=\frac1{2^{n+1}}$ at  step $n+1$ because on the subcylinder $[u0]=[u_1\dots 
u_n 0]$ the random walk will move from the position $x=0$ to the position $x=-1$ at this step.\\% Similarly, for $n\in2\mathbb{N}+1$, any path of the random walk $u=u_1\dots u_n\in\{0,1\}^\star$ with $b(u)=-1$ and $n\leq 
%\eta_w(u)$ will increase  the value of $\PP(S_n^w\in\N)$ by $\frac12\cdot \PP([u])=\frac1{2^{n+1}}$ at the  step $n+1$. We will express it more generally in Proposition \ref{Azyan}. \\

%The events that govern the above exchanges of mass are described by the
%position of the walk \emph{together with} its survival. 

Given $w\in\{0,1\}^\star$, for $m\ge0$ and
$x\in\Z$ define
\begin{equation}\label{eq:rho}
\rho^{w}_m(x):=\;\PP\bigl[\tau_w>m,\ S_m=x\bigr]
\;=\;\PP\bigl[\eta_w\ge m,\ S_m=x\bigr],
\end{equation}
the probability that after $m$ steps the walk still occupies an internal
node of $T_w$, located at the position $x\in \mathbb{Z}$ ( and $m$ is the depth of the node in the tree). Naturally, if $x \not\equiv m \pmod 2$, then $\PP[S_m=x]=0$, and hence

\begin{equation}\label{eq:parity}
\rho^{w}_m(x)=0\qquad\text{for any } x\not\equiv m\ (\mathrm{mod}\ 2).
\end{equation}

 To sum up, the running mass of $\N$ decreases only at
the passages from an even to an odd time (by halving  the probability of
surviving at the position $0$), and increases only at the passages from
an odd to an even time ( by half of the probability of surviving at the
position $-1$),  see Proposition \ref{Azyan} below. 
%Note that the event $\{\tau_w>m\}$ is the disjoint union of the events
%$[u]$ over the words $u\in\{0,1\}^{m}$ which are internal nodes of
%$T_w$, and $S_m=b(u)$ on $[u]$; hence $2^{m}\rho^{w}_m(x)$ is the
%number of internal nodes of $T_w$ at depth $m$ and position $x$.
 
\begin{proposition}\label{Azyan}
\label{prop:onestep}
For every $w\in\{0,1\}^{\star}$ and every $m\ge0$,
\begin{equation}\label{eq:onestep}
\PP\bigl[S^{w}_{m+1}\ge0\bigr]-\PP\bigl[S^{w}_{m}\ge0\bigr]
\;=\;-\tfrac12\,\rho^{w}_m(0)\;+\;\tfrac12\,\rho^{w}_m(-1).
\end{equation}
In particular, by \eqref{eq:parity}, for every $n\in\N$,
\begin{equation}\label{eq:twosteps}
\begin{aligned}
\PP\bigl[S^{w}_{2n+1}\ge0\bigr]
&=\PP\bigl[S^{w}_{2n}\ge0\bigr]-\tfrac12\,\rho^{w}_{2n}(0),\\
\PP\bigl[S^{w}_{2n+2}\ge0\bigr]
&=\PP\bigl[S^{w}_{2n+1}\ge0\bigr]+\tfrac12\,\rho^{w}_{2n+1}(-1).
\end{aligned}
\end{equation}
\end{proposition}
 
\begin{proof}
We need to prove \eqref{eq:onestep}. On $\{\tau_w\le m\}$ we have $S^{w}_{m+1}=S^{w}_{m}$, so this event does
not contribute to the increment. On the event $\{\tau_w>m\}$ we have
$$S^{w}_{m}=S_m\mbox{ and }S^{w}_{m+1}=S_m+\xi_{m+1}.$$

For $x\in\Z$, the event 
$$A^m_x:=\{\tau_w>m,\ S_m=x\}$$ is independent of $\xi_{m+1}$, and 
$$ \{\tau_w>m\}=\bigcup_{x\in\Z}A^m_x\mbox{ and }\PP(A^m_x)=\rho^{w}_m(x)$$ 
The characteristic functions of $\{S^{w}_{m+1}\ge0\}$ and $\{S^{w}_{m}\ge0\}$ may differ only on $A^m_{0}$ and $A^m_{-1}$:
$$\PP\bigl[S^{w}_{m+1}\ge0\bigr]-\PP\bigl[S^{w}_{m}\ge0\bigr]=\PP[A^m_{-1}, \xi_{m+1}=+1]-\PP[A^m_{0}, \xi_{m+1}=-1]=$$
$$\frac12\PP[A^m_{-1}]-\frac12\PP[A^m_{0}]=-\tfrac12\,\rho^{w}_m(0)\;+\;\tfrac12\,\rho^{w}_m(-1).$$
\end{proof}

From the above:

$$ \PP[S^w_{2n+1}\geq 0] \leq  \PP[S^w_{2n}\geq0] \mbox{ and } \PP[S^w_{2n+2}\geq 0 ] \geq  \PP[S^w_{2n+1}\geq0 ],\ n\in\N.$$ 
\ \\

Now, we aim to prove the  analogue of Proposition 4.2 from \cite{Cheng} which was formulated therein in the context of deletion-closed languages.  From Proposition \eqref{Azyan}, by telescoping,

\begin{corollary}\label{cor:telescope}
For every $w\in\{0,1\}^{\star}$,
\begin{equation}\label{eq:telescope}
P_w(\N)\;=\;1\;-\;\frac12\sum_{n\ge0}
\Bigl(\rho^{w}_{2n}(0)-\rho^{w}_{2n+1}(-1)\Bigr)
\;=\;\frac12\;+\;\frac12\sum_{n\ge0}
\Bigl(\rho^{w}_{2n+1}(-1)-\rho^{w}_{2n+2}(0)\Bigr).
\end{equation}
As $\rho^{w}_m=0$ for $m>\ell(w)$, the above sums are finite. 
\end{corollary}

\begin{proof}
 The first equality is a direct conclusion from Proposition \ref{Azyan}. The second equality follows from $\rho^{w}_0(0)=\PP[\tau_w>0]=1$. 
\end{proof}

\begin{remark} $P_w(\N)\ge\tfrac12$ if and only if
$\sum_{n\ge0}\bigl(\rho^{w}_{2n}(0)-\rho^{w}_{2n+1}(-1)\bigr)\le1$, and
the same equivalence holds with both inequalities strict.
\end{remark}

Define

\begin{equation}\label{eq:pathlevels}
\widetilde{B}_n^{x}(w)
\;=\;\bigl\{u\in\{0,1\}^{n}:\ [u]\subseteq\{\tau_w>n\},\ b(u)=x\bigr\}
%\;=\;\bigl\{u\in N(T_{w}):\ \ell(u)=n,\ b(u)=x
\bigr\},
\end{equation}

and note that, since $\PP([u])=\frac1{2^n}$ for any $u$ from $\widetilde{B}_n^{x}(w)$, we have
\begin{equation}\label{eq:nodecount}
\rho^{w}_n(x)\;=\frac1{2^n}\,\bigl|\widetilde{B}_n^{x}(w)\bigr|,
\qquad n\ge0,\ x\in\Z .
\end{equation}
Now, define
\begin{equation}\label{eq:pathdefects}
\widetilde{\Delta}^{-}_m(w)
\;=\;2\,\bigl|\widetilde{B}_{2m-1}^{-1}(w)\bigr|
-\bigl|\widetilde{B}_{2m}^{0}(w)\bigr|,
\qquad m\ge1 .
\end{equation}

By Corollary \ref{cor:telescope},
 \begin{corollary}\label{cor:pathexpansion}
For every $w\in\{0,1\}^{\star}$,
\begin{equation}\label{eq:pathexpansion}
P_w(\N)\;=\;\frac12+\sum_{m\ge1}(\frac12)^{2m+1}\cdot \,
\widetilde{\Delta}^{-}_m(w),
\end{equation}
the sum having finitely many nonzero terms.
\end{corollary}
Corollary \ref{cor:pathexpansion} holds true for a random walk on any binary tree - it is an analogue of Proposition 4.2 from \cite{Cheng}.  In case of the family $\{T_w\}_{w\in\{0,1\}^\star}$, by Lemma 4.3 in \cite{Cheng}), it is already known that $\widetilde{\Delta}^{-}_m(w)\geq 0$ for any $m\in \N_{\geq 1}$. In Section \ref{criterion}, this inequality  will be  established as the conclusion  from Theorem 5.2 . Having known that  $\widetilde{\Delta}^{-}_m(w)\geq 0$, in order to prove the Theorem \ref{main} we will need to show that for any $w\in\{0,1\}^\star$ the following conditions are equivalent:
\begin{enumerate}
\item $w$ is saturated in sense of  Definition 2.1 
\item for any $m\geq  1$, $\widetilde{\Delta}^{-}_m(w)= 0$.
\end{enumerate}
%The next section introduces the combinatorial machinery of 

\section{The subword order and the node language of $T_t$}
\label{sec:nodes}
 In \cite{Tarlowski}, the stopping times $\tau_w$, $w\in\{0,1\}^\star$, have been defined recursively. Paper  \cite{Cheng} has introduced the combinatorial description of the trees $T_w$ that enables an explicit description of the stopping times $\tau_w$. In this section, we introduce the node-enumeration from \cite{Cheng}, and also present some of its consequences. For the reader's convenience, this section is self-contained and includes proofs of all the statements presented.
%In this section we introduce the combinatorial apparatus of Cheng
%cite{Cheng}: the scattered--subword order, the level sets of subwords,
%the insertion slots, and the node language of a tree. The main results of
%the section identify the node language $N(T_t)$ with a principal ideal of
%the subword order and express the killed occupations \eqref{eq:occdef}
%as weighted counts of subwords of $w$.

\begin{definition}[Subword order]\label{def:subword}
For $u,a\in\{0,1\}^{\star}$ with $u=u_1\dots u_n$ and $a=a_1\dots a_N$ we
will say that $u$ a \emph{subword} of $a$, denoted by $u\sub a$, if there exist
indices $1\le i_1<i_2<\dots<i_n\le N$ such that $u_j=a_{i_j}$ for any
$1\le j\le n$; such a sequence of indices is called an \emph{embedding} of
$u$ into $a$. We set
\[
\Sub(a)\;=\;\{u\in\{0,1\}^{\star}:\;u\sub a\}.
\]
\end{definition}

 The relation $\sub$ is a partial order, and $u\in \Sub(a)$ iff and $u$ is
obtained from $a$ by deleting some of the digits (in particular, $\eps\in \Sub(a)$). Obviously:
\begin{equation}\label{eq:monotone}
u\sub a\ \Longrightarrow\ \lz(u)\le\lz(a)\ \text{ and }\ \lo(u)\le\lo(a).
\end{equation}
Additionally, is easy to note that for any $a\in\{0,1\}^{\star}$,
\begin{align}
\{u:\ 0u\sub 0a\}&=\Sub(a), &
\{u:\ 1u\sub 0a\}&=\{u:\ 1u\sub a\},\label{eq:prependL}\\
\{u:\ 1u\sub 1a\}&=\Sub(a), &
\{u:\ 0u\sub 1a\}&=\{u:\ 0u\sub a\}.\label{eq:prependR}
\end{align}

%All counting in this paper is organized by the \emph{levels} of the
%subword order
 For $w\in\{0,1\}^{\star}$, $n\ge0$ and $x\in\Z$ let
\begin{equation}\label{eq:levelsets}
B_n^{x}(w)\;=\;\bigl\{u\sub w:\ \ell(u)=n,\ b(u)=x\bigr\}.
\end{equation}
 A word of length $n$ and balance $x$ has $\lz=\frac{n-x}{2}$ zeros and
$\lo=\frac{n+x}{2}$ ones. Additionally, if either $|x|>n$ or $ x\not\equiv n\ (\mathrm{mod}\ 2)$, then
\begin{equation}\label{eq:levelparity}
B_n^{x}(w)=\varnothing. \end{equation}

The reversal map $u\mapsto\rev{u}$ is a bijection between $B_n^{x}(w)$ and $B_n^{x}\bigl(\rev{w}\bigr)$, and the complement map $u\mapsto\mir{u}$ is a bijection between $B_n^{x}(w)$ and $B_n^{-x}\bigl(\mir{w}\bigr)$. Hence:

$$|B_n^{x}(w)|=|B_n^{x}(\rev{w})|=|B_n^{-x}(\mir{w})|.$$

Finally, given $m\geq 1$ and $w\in\{0,1\}^\star$:
\begin{equation}\label{eq:defects}
\Delta^-_m(w)\;:=\;2\,\bigl|B_{2m-1}^{-1}(w)\bigr|-\bigl|B_{2m}^{0}(w)\bigr|,
%\qquad
%\Delta^+_m(w)\;=\;2\,\bigl|B_{2m-1}^{+1}(w)\bigr|-\bigl|B_{2m}^{0}(w)\bigr|,
%\qquad m\ge1 .
\end{equation}
%By \eqref{eq:levelsym}, $\Delta^{\pm}_m(\rev{w})=\Delta^{\pm}_m(w)$ and
%$\Delta^{+}_m(w)=\Delta^{-}_m(\mir{w})$ for all $m$.

\begin{definition}[Insertion slots]\label{def:slots}
Let $v=v_1\dots v_n\in\{0,1\}^{\star}$. The word $v$ has $n+1$ \emph{insertion
slots} - positions indexed by $j\in\{0,1,\dots,n\}$. Inserting a digit $X\in\{0,1\}$ in
slot $j$ produces the word
\[
\iota^{X}_{j}(v):=\;v_1\dots v_j\,X\,v_{j+1}\dots v_n
\]
of length $n+1$, in which the inserted digit occupies position $j+1$.
\end{definition}

Above, one may note that the  insertion of a digit in two different slots of the given word may result in producing the same word.\\

 Given a digit $X\in\{0,1\}$ and $A\subset\{0,1\}^\star$,\\
$$X\,A:=\{Xu:\ u\in A\}\mbox{ ( in particular},\ X\varnothing=\varnothing.)$$
The following definition allows to identify a binary tree with the set of its internal nodes.

\begin{definition}[Node language]\label{def:nodelanguage}
The \emph{node language} $N(T)\subseteq\{0,1\}^{\star}$ of a tree
$T\in\mathcal{T}$ is defined recursively by
\[
N(\bullet)=\varnothing,\qquad
N([T^-,T^+])=\{\eps\}\;\cup\;0\,N(T^-)\;\cup\;1\,N(T^+),
\]

\end{definition}
 Given any $T\in\mathcal{T}$,  $N(T)$ is the set of words labelling
the internal nodes of $T$ (in particular,   the root of the tree is labelled by $\eps$), and the leafs of the tree are not labelled. For instance, 
$$N(T_{\eps})=N(T_{(11)_2})=N([\bullet,\bullet])=\{\eps\}\mbox{ and }N(T_{(111)_2})=N([\bullet, [\bullet,\bullet]])=\{\eps,1\}.$$ 
Recall that $[\eps]=\Omega$. We write: $\bigcup\limits_{\substack{u\in \varnothing}}[u]=\varnothing$. From the recursive definitions of
$\tau_T$ and $N(T)$:
\begin{equation}\label{eq:cylinder}
\{\tau_T>n\}\;=\;\bigcup_{\substack{u\in N(T)\\ \ell(u)=n}}[u], n\geq0.
\end{equation}
Equivalently, as the set of labelling $N(T)$ is prefixed-closed by definition, we have:
\begin{equation}\label{eq:firstexit}
\tau_T(\omega)\;=\;\min\bigl\{n\ge0:\ \omega_1\dots \omega_n\notin N(T)\bigr\},\ \omega\in\Omega.
\end{equation}

%Equation \eqref{eq:cylinder} partitions the event $\{\tau_T > n\}$ according to
%the first $n$ steps of the walk: the walk has not yet stopped after $n$ steps
%if and only if the word encoding these steps labels an internal node of $T$ at
%depth $n$.

\begin{proposition}[{\cite{Cheng},Proposition 3.3}]\label{Chengis}
\label{Nodes}
For every $w\in\{0,1\}^{\star}$,
\begin{equation}\label{eq:nodelanguage}
N\bigl(T_{w}\bigr)\;=\;\Sub\bigl(\rev{w}\bigr).
\end{equation}
\end{proposition}

\begin{proof}
Induction on $\ell(w)$. For $w=\eps$ we have $T_{\eps}=[\bullet,\bullet]$ and
$N(T_{\eps})=\{\eps\}=\Sub(\eps)$.

Now, note that, by Definition~\ref{def:nodelanguage}, for $T=[T^-,T^+]$,
\begin{equation}\label{eq:subtreelanguages}
N(T^-)=\{u:\ 0u\in N(T)\},
\qquad
N(T^+)=\{u:\ 1u\in N(T)\}.
\end{equation}
Now, for the induction step assume that  $w$ satisfy \eqref{eq:nodelanguage}. Put $a=\rev{w}$ so we have:
$$N(T_w)=\Sub(a).$$
By  the growth recursion \eqref{growth}, $T_{w0}=[T_w,T_w^+]$. Hence by
Definition~\ref{def:nodelanguage} and \eqref{eq:subtreelanguages},
\[
N(T_{w0})
=\{\eps\}\cup 0\,N(T_w)\cup 1\,N(T_w^{+})
=\{\eps\}\cup 0\,\Sub(a)\cup 1\,\{u:\ 1u\sub a\}.
\]
On the other hand, decomposing $\Sub(0a)$ according to the first digit:
\[
\Sub(0a)=\{\eps\}\cup 0\,\{u:\ 0u\sub 0a\}\cup 1\,\{u:\ 1u\sub 0a\}
=\{\eps\}\cup 0\,\Sub(a)\cup 1\,\{u:\ 1u\sub a\}.
\]
Since $\rev{w0}=0\,\rev{w}=0a$, this proves \eqref{eq:nodelanguage} for
$w0$. The argument for $w1$ is analogous.
\end{proof}

The above proposition provides a compact combinatorial description of $\tau_u$.  Indeed, by \eqref{eq:firstexit}, it leads to:

\begin{corollary}
For any $u\in\{0,1\}^\star$, 
\begin{equation}\label{tauw}
\tau_u(w)=\inf \{n\geq 1\colon w_1\dots w_n\nsub \rev{u}\}
\end{equation}
\end{corollary}
By Proposition \ref{Chengis} and  \eqref{eq:cylinder}, and by $|B_n^x(w)|=|B_n^x(\rev{w})|$,
\begin{corollary}\label{BB}
For any $w\in\{0,1\}^\star$, $x\in\Z$, $n\geq0$, 
$$\widetilde{B}_n^x(w)=B_n^x(\rev{w}),\mbox{ and hence }, \widetilde{\Delta}^-_m(w)=\Delta^-_m(w).$$
\end{corollary}

By Corollary~\ref{cor:pathexpansion} and Corollary~\ref{BB}: 

\begin{proposition}[{ \cite[Proposition 4.2]{Cheng}}]\label{Cheng}
For every $w\in\{0,1\}^{\star}$,
\begin{equation}\label{eq:expansion}
P_w(\N)=\frac12+\sum_{m\ge1}2^{-2m-1}\,\Delta^-_m(w).
%\qquad
%P_t(-\N)=\frac12+\sum_{m\ge1}2^{-2m-1}\,\Delta^+_m(w),
\end{equation}

\end{proposition}
At the end of this section, recall that the stopping time $\tau_u(w)$ returns the number of steps after which the path $w$ reaches any leaf of the tree $T_u$. Given $u\in\{0,1\}^\star$, let
\begin{equation}\label{exit}
\Lambda(u)=\bigl\{v\in\{0,1\}^\star:\ \ell(v)\ge1,\
v_1\dots v_{\ell(v)-1}\sub\rev{u},\ v\nsub\rev{u}\bigr\}
\end{equation}
denote the set of words labelling the leaves of the tree $T_u$. By definition, measures $P_u$ are given by a compact formula:

\begin{corollary}\label{cor:explicitP}
For any $u\in\{0,1\}^\star$, by $P_u=\mathcal{L}(S_{\tau_u})$,
\begin{equation}\label{explicitP}
P_u=\sum_{v\in\Lambda(u)}\Bigl(\frac12\Bigr)^{\ell(v)}\dd_{b(v)} .
\end{equation}
\end{corollary}

\begin{proof}
It is enough to note that $b(v)$ is the position of the walk at the end of a path $v\in\{0,1\}^\star$, and $(\frac{1}{2})^{\ell(v)}$ is the probability of this path.
\end{proof}

\section{Properties of $T_w$, and the formula for $P_w(\N)$.}\label{criterion}
In this section we  express explicitely  the relation between $|B^{x-1}_{n-1}|$ and $|B^x_n|$. The case $x=0$ is of special interests - it allows to control the value of $\Delta_m^-$. In the proof, as in \cite{Cheng}, we count the words with an inserted/deleted digit.\\

 Fix $w\in\{0,1\}^{\star}$ so we may write shortly $B_n^{x}=B_n^{x}(w)$, $\Delta^{-}_m=\Delta^{-}_m(w)$,
$\rho_n(x)=\rho^{w}_n(x)$.  For $n\ge1$ and $x\in\Z$ let
\begin{equation}\label{eq:failing}
F_1(n,x)\;=\;|\bigl\{(v,q):\ v\in B_{n-1}^{x-1},\
0\le q\le n-1,\ \iota^{1}_{q}(v)\nsub w\bigr\}|.
\end{equation}
The $F_1(n,x)$ counts all the situations in which insertion a digit $1$ to some slot of some word $v\in B_{n-1}^{x-1}$ produces a word outside $\Sub(w)$.

\begin{proposition}\label{prop:markedcount}
For every $n\ge1$ and $x\in\Z$,
\begin{equation}\label{eq:markedcount}
\frac{n+x}{2}\,\bigl|B_n^{x}\bigr|\;+\;F_1(n,x)
\;=\;n\,\bigl|B_{n-1}^{x-1}\bigr| .
\end{equation}
\end{proposition}

\begin{proof}
If $x\not\equiv n\pmod2$, all three terms vanish by
\eqref{eq:levelparity}. The same in the case $|x|>n$. Assume that
$x\equiv n\pmod2$, and that $|x|\leq n$. Every $u\in B_n^{x}$ has exactly
$\lo(u)=\frac{n+x}{2}$ ones. Consider the set of marked
words
\[
M\;=\;\bigl\{(u,i):\ u\in B_n^{x},\ 1\le i\le n,\ u_i=1\bigr\},
\]
in which the second coordinate marks one of the ones of $u$, so
that $|M|=\frac{n+x}{2}\,|B_n^{x}|$, and define the set of admissible
marked slots
\[
I\;=\;\bigl\{(v,q):\ v\in B_{n-1}^{x-1},\ 0\le q\le n-1,\
\iota^{1}_{q}(v)\sub w\bigr\}.
\]
For the proof it will be enough to note that $|M|=|I|$, which follows from the fact  that the following two mappings:
\[
\delta\colon M\ni (u,i)\longmapsto\bigl(u_1\dots u_{i-1}u_{i+1}\dots u_{n},\;
i-1\bigr) \in I, 
\qquad
\gamma\colon  I\ni (v,q)\longmapsto \bigl(\iota^{1}_{q}(v),\;q+1\bigr)\in M
\]
are mutually inverse bijections between $M$ and $I$: the $\gamma$ insert a digit $1$ into a word $v\in B^{x-1}_{n-1}$ and the composition $\delta\circ\gamma$  deletes the digit $1$ which was just
inserted (and records the same slot at the end), and $\gamma\circ\delta$ re-inserts the
deleted digit into its original position ( returning the same slot at the end ). Since every $v\in B_{n-1}^{x-1}$ has length $n-1$ and therefore exactly
$n$ slots, the definition \eqref{eq:failing} gives $|I|=n\,|B_{n-1}^{x-1}|-F_1(n,x)$. By $|I|=|M|=\frac{n+x}{2}|B_n^x|$, we have \eqref{eq:markedcount}.
\end{proof}

\begin{theorem}\label{thm:domination}
For every $w\in\{0,1\}^{\star}$ and every $n\ge1$,
\begin{equation}\label{eq:domination}
\rho_n(x)\ \le\ \rho_{n-1}(x-1),\ x\ge0,\mbox{ \  and \ } \rho_n(x)\ \le\ \rho_{n-1}(x+1),\ x\le0.
\end{equation}
Equivalently,
\begin{equation}\label{eq:dominationcount}
\bigl|B_n^{x}\bigr|\ \le\ 2\,\bigl|B_{n-1}^{x-1}\bigr|\quad(x\ge0),
\qquad
\bigl|B_n^{x}\bigr|\ \le\ 2\,\bigl|B_{n-1}^{x+1}\bigr|\quad(x\le0).
\end{equation}
\end{theorem}

\begin{proof}
The two equations  are equivalent by \eqref{eq:nodecount} and Corollary~\ref{BB}. If
$x\not\equiv n\pmod2$ then $B^n_x=\varnothing$. Assume that $x\equiv n\pmod2$, and consider the case $x\ge0$.
By \eqref{eq:markedcount},
$$
\frac{n+x}{2}\,\bigl|B_n^{x}\bigr|\leq\;n\,\bigl|B_{n-1}^{x-1}\bigr| , $$
 and, by $n\geq1$ and $x\geq0$,
\[
\bigl|B_n^{x}\bigr|\ \le\ \frac{2n}{n+x}\,\bigl|B_{n-1}^{x-1}\bigr|
\ \le\ 2\,\bigl|B_{n-1}^{x-1}\bigr|.
\]
 This
proves the first inequality in \eqref{eq:dominationcount}. The second
follows by considering the complement $\mir{w}$:
$$|B_n^{x}(w)|=|B_n^{-x}(\mir{w})|\leq 2\,|B_{n-1}^{-x-1}(\mir{w})|=2\,|B_{n-1}^{x+1}(w)|\mbox{ for }x\le0.$$
\end{proof}

\begin{corollary}\label{cor:defects}
For every $w\in\{0,1\}^{\star}$ and $m\ge1$:
\begin{enumerate}[label=\textup{(\alph*)}]
\item\label{def:nonneg} $\rho_{2m}(0)\le\rho_{2m-1}(-1)$, and hence $\Delta^-_m(w)\geq 0$
\item\label{def:exact} $m\,\Delta^-_m(w)=F_1(2m,0)$, and hence $\Delta^-_m(w)=0\Leftrightarrow F_1(2m,0)=0$.
\end{enumerate}
%The complementary statements hold for $\Delta^+_m(w)$, with the digit $1$
%replaced by $0$ in the marked count.
\end{corollary}

\begin{proof}
Let $(n,x)=(2m,0)$. Theorem~\ref{thm:domination} implies  \ref{def:nonneg}. Equation \ref{def:exact} follows from \eqref{eq:markedcount} and the definiton of $\Delta_{m}^-(w)$.
\end{proof}

By Proposition~\ref{Cheng} and Corollary~\ref{cor:defects}, we have $P_w(\N)\geq\frac12$,  and:
\begin{corollary}
$P_w(\N)=\frac12$ iff  $F_1(2m,0)=0$ for any $m\geq1$.
\end{corollary}

The above means that $P_w(\N)=\frac12$ iff for any $m\geq1$ with $B_{2m-1}^{-1}(w)\neq\varnothing$, any word $v\in B_{2m-1}^{-1}(w)$ is 
such that $\iota^{1}_{j}(v)\sub w$ for every slot $j\in\{0,1,\dots,2m-1\}$ - inserting a digit $1$ to any slot of any $v\in B_{2m-1}^{-1}(w)$ produces a subword of $w$. Also, note that $v\in B_{2m-1}^{-1}(w)$ for some $m\geq 1$ iff $v$ is a subword of $w$ with $\lz(v)=\lo(v)+1$. Hence:

\begin{corollary}\label{cor:star}
$P_w(\N)=\frac12$ if and only if for any $v\sub w$ with $\lz(v)=\lo(v)+1$, every insertion of a single digit $1$ into $v$ produces a subword of $w$.

\end{corollary}
%\end{equation}
%Moreover, if \eqref{eq:star} fails for some $v$ with $\lz(v)=m$, then
%$P_t(\N)\ge\frac12+2^{-2m-1}$.
%\end{corollary}

%\begin{corollary}
%$P_w(\N)=\frac12$ iff  $F_1(2m,0)=0$ for any $m\geq1$. In words: $P_w(\N)=\frac12$ iff for any $m\geq1$ with $B_{2m-1}^{-1}(w)\neq\varnothing$, any word $v\in B_{2m-1}^{-1}(w)$ is 
%such that $\iota^{1}_{j}(v)\sub w$ for every slot $j\in\{0,1,\dots,2m-1\}$ (inserting a digit $1$ to any slot of $v$ produces a subword of $w$ )
%\end{corollary}

%\begin{proposition}[Insertion criterion]\label{prop:criterion}
%Let $w\in\{0,1\}^{\star}$ and $m\ge1$. Then:\\
%
%
%\label{crit:strict} if some pair $(v,j)$ with
%$v\in B_{2m-1}^{-1}(w)$ and $0\le j\le 2m-1$ satisfies
%$\iota^{1}_{j}(v)\nsub w$, then $\Delta^-_m(w)\ge1$.
%
%The complementary statements hold for $\Delta^+_m(w)$, with the roles of the digits $0$
%and $1$ interchanged.
%%\end{proposition}
%
%\begin{proof}
%\ref{crit:nonneg} is Corollary~\ref{cor:defects}\ref{def:nonneg}.
%\ref{crit:eq}: by Corollary~\ref{cor:defects}\ref{def:exact},
%$\Delta^-_m(w)=0$ if and only if $F_1(2m,0)=0$, i.e.\ if and only if no
%pair $(v,j)$ with $v\in B_{2m-1}^{-1}(w)$ fails.
%\ref{crit:strict}: one failing pair gives $F_1(2m,0)\ge1$, hence
%$m\,\Delta^-_m(w)=F_1(2m,0)\ge1$; as $\Delta^-_m(w)$ is an integer,
%$\Delta^-_m(w)\ge1$.
%\end{proof}

\begin{example}\label{ex:LR}
Let $w=01$. The subwords of $w$ are: $\eps,0,1,01$. We have:
$$P_w(\N)=P_{(1w1)_2(\N)}=P_{11}(\N)=\frac58>\frac12,$$
because the word $v=0$ satisfies $v\sub w$, and $l_0(v)=l_1(v)+1$, but $v1=01\nsub 01.$

\end{example}

\begin{example}\label{ex:RLR}
Let $w=101$, $t=27$. Here: $\Sub(w)=\{\eps,0,1,01,10,11,101\}$, and $P_{27}(\N)=\frac12$.
\end{example}

Before we move to the proof of Theorem \ref{main}, we will express how the words $v$ which fails the property from Corollary \ref{cor:star} increase the total mass of $\N$ - this will provide the lower bound for $P_w(\N)$ in case of nonsaturated words.\\

For $w\in\{0,1\}^{\star}$ and $m \ge 1$ define the set of deficient words  at level $m$:
\[
D_m(w) \;=\; \bigl\{\, v \in B^{-1}_{2m-1}(w) \;:\; \iota^{1}_{q}(v) \nsub w
\ \text{for some } q \in \{0,1,\dots,2m-1\} \,\bigr\},
\qquad d_m(w) = \bigl|D_m(w)\bigr|.
\]

\begin{theorem}\label{thm:cumulative}
For every $w \in \{0,1\}^{\star}$,
\begin{equation}\label{eq:exact-FR}
P_w(\N) \;=\; \frac12 \;+\; \sum_{m \ge 1} \frac{2^{-2m-1}}{m}\, F_1(2m,0),
\end{equation}
and hence
\begin{equation}\label{eq:cumulative}
P_w(\N) \;\ge\; \frac12 \;+\; \sum_{m \ge 1} 2^{-2m-1}
\left\lceil \frac{d_m(w)}{m} \right\rceil .
\end{equation}
\end{theorem}

\begin{proof}
 The identity \eqref{eq:exact-FR} follows from Proposition~\ref{Cheng}:
\begin{equation}\label{eq:median-expansion-again}
P_w(\N) \;=\; \frac12 + \sum_{m \ge 1} 2^{-2m-1}\, \Delta^-_m(w),
\end{equation}
and Corollary~\ref{cor:defects}\,\ref{def:exact}: 
\begin{equation}\label{eq:mDelta}
m\, \Delta^-_m(w) \;=\; F_1(2m,0), m\geq 0.
\end{equation}

Now, by definition,  the map 
$$F_1(2m,0)\ni(v,q) \mapsto v\in D_m(w)$$ is a surjection, and hence
\begin{equation}\label{eq:FRged}
F_1(2m,0) \;\ge\; d_m(w), \qquad m \ge 1.
\end{equation}
Combining \eqref{eq:mDelta} and \eqref{eq:FRged} gives
$\Delta^-_m(w) \ge d_m(w)/m$. Since $\Delta^-_m(w)$ is an integer ( \eqref{eq:defects}), we get:
\begin{equation}\label{eq:ceiling}
\Delta^-_m(w) \;\ge\; \left\lceil \frac{d_m(w)}{m} \right\rceil, \qquad m \ge 1.
\end{equation}

\end{proof}

\begin{corollary}\label{star}
 If  for some $v\sub w$ with $\lz(v)=m=\lo(v)+1$, an insertion of a single digit $1$ into some slot of $v$ produces a word that is not a subword of $w$, then
$P_t(\N)\ge\frac12+2^{-2m-1}$.
\end{corollary}

\begin{proof}
 By the assumption, $v\in D_m(w)$, so $d_m(w)\ge1$
and \eqref{eq:cumulative} gives
\[
P_t(\N)\;\ge\;\frac12+2^{-2m-1}\left\lceil\frac{d_m(w)}{m}\right\rceil
\;\ge\;\frac12+2^{-2m-1}.
\qedhere
\]
\end{proof}

\section{Proof of Theorem \ref{main}}\label{sec:proof}

We start with twe following simple observation - the proof is an easy exercise.

\begin{observation}\label{greedy} Assume that $w\in\{0,1\}^{\star}$ has $k=\lz(w)$ zeros and block profile $(a_0,\dots,a_k)$, that is, $w=1^{a_0}0\,1^{a_1}\cdots 0\,1^{a_k}$.
\begin{enumerate}
\item  If $a_i\ge k$ for every $0\le i\le k$, then every word $u$ with
$\lz(u)\le k$ and $\lo(u)\le k$ satisfies $u\sub w$.
\item  Assume that some $u\in\{0,1\}^\star$ with $\lz(u)=\lz(w)=k\ge1$ has block profile $(c_0,\dots,c_k)$, i.e.\
$u=1^{c_0}0\,1^{c_1}\cdots 0\,1^{c_k}$. Then
\[
u\sub w
\quad\Longleftrightarrow\quad
c_j\le a_j\ \text{ for every }0\le j\le k .
\]
\end{enumerate}

\end{observation}

We will need the following lemma.

\begin{lemma}\label{lem:deficiency}
Let $w\in\{0,1\}^{\star}$ be non-saturated and put
$m^{\ast}=\min(\lz(w),\lo(w)+1)$. Then there exist a word $v\sub w$ with
$\lz(v)=m^{\ast}$, $\lo(v)=m^{\ast}-1$ and a slot
$q\in\{0,1,\dots,2m^{\ast}-1\}$ such that $\iota^{1}_{q}(v)\nsub w$.
\end{lemma}

\begin{proof}
Since $w$ is non-saturated, $k:=\lz(w)\ge1$ (every word with $k=0$ is
saturated) and there is an index $r\in\{0,\dots,k\}$ with 
$$a_r\le k-1,$$
where $(a_0,\dots,a_k)$ is the block profile of $w$ as in \eqref{eq:profile}.
We distinguish two cases.

\smallskip
\emph{Case 1: $\lo(w)\ge k-1$.}
Define a profile $(c_0,\dots,c_k)$ as follows: let $c_r=a_r$, and now for
$i\in\{0,1,\dots,k\}\setminus\{ r\}$ choose any integers $0\le c_i\le a_i$
for which we have
\[
\sum_{i\ne r}c_i\;=\;k-1-a_r .
\]
Such a choice exists because $0\le k-1-a_r$ (by $a_r\le k-1$) and
$ k-1-a_r\leq \lo(w)-a_r=\sum\limits_{i\ne r}a_i$ (by the assumption:
$\lo(w)\ge k-1$ ). Let $v$ be the word with $k$ zeros and
block profile $(c_0,\dots,c_k)$; then $\lz(v)=k$, $\lo(v)=\sum_i c_i=k-1$,
and $v\sub w$ by point (2) of Observation \ref{greedy}. Now insert a digit
$1$ into the $r$-th block of ones in $v$ ( insert $1$ in slot $0$ if $r=0$,
or, if $r\ge1$, insert $1$ immediately after the $r$-th zero in $v$).
The resulting word $v'$ has $k$ zeros and $k$ ones, and has
block profile $(c_0,\dots,c_{r-1},\,c_r+1,\,c_{r+1},\dots,c_k)$ with
$c_r+1=a_r+1>a_r$, so $v'\nsub w$ by point (2) of Observation \ref{greedy}.
Here, $\lz(v)=k$, and since $\lo(w)\ge k-1$ we have
$k=\min(\lz(w),\lo(w)+1)=:m^{\ast}$.

\smallskip
\emph{Case 2: $\lo(w)\le k-2$.}
Put $m^\star=\lo(w)+1\le k-1$; then $m^\star=\min(\lz(w),\lo(w)+1)$. Let $v$ be the subword of $w$ created by
deleting arbitrary $(k-m^\star)$ zeros from the word $w$: we save all the
ones of $w$ and exactly $m^\star$ zeros; then $v\sub w$,
$\lz(v)=m^\star$ and $\lo(v)=m^\star-1$. Now, it is enough to note that every word $v'$
obtained from $v$ by inserting a digit $1$ to any slot satisfies
$\lo(v')=m^\star>\lo(w)$, and hence $v'\nsub w$.
\end{proof}

By Lemma \ref{lem:deficiency} and Corollary \ref{star},

\begin{corollary}\label{cor:main-ii}
If $w \in \{0,1\}^{\star}$ is not saturated, then
\[
P_{w}(\N) \;\ge\; \frac12 \;+\; 2^{-2\min(\lz(w),\,\lo(w)+1)-1}.
\]
\end{corollary}

%\begin{observation}[Pinning]\label{lem:pinning}

%\end{observation}

%\begin{proof}
%In Lemma~\ref{lem:embed} with $j=k$, the only choice of indices with
%$1\le i_1<\dots<i_k\le k$ is $i_s=s$ for all $s$, and then each sum in
%\eqref{eq:capacity} reduces to a single term:
%$\sum_{i<i_1}a_i=a_0$, $\sum_{s\le i<s+1}a_i=a_s$, and
%$\sum_{i\ge k}a_i=a_k$.
%\end{proof}

%\begin{proof}
%Let $j=\lz(u)\le k$ and choose $i_s=s$ for $1\le s\le j$ in
%Lemma~\ref{lem:embed}. Every left-hand side in \eqref{eq:capacity} is at most
%$\lo(u)\le k$, while each right-hand side is at least one of the $a_i\ge k$;
%for the last inequality note $\sum_{i\ge j}a_i\ge a_j\ge k$. Hence
%\eqref{eq:capacity} holds and $u\sub w$.
%\end{proof}

\begin{proof}[Proof of Theorem \textup{\ref{main}}]

Assume that $w\in\{0,1\}^{\star}$ has $k=\lz(w)$ zeros and
block profile $(a_0,\dots,a_k)$ as in \eqref{eq:profile}. Note first that the word $w=\eps$
is saturated and the measure $P_{\eps}=P_3=\frac12\dd_{-1}+\frac12\dd_{1}$
satisfies $P_{\eps}(\N)=\frac12$. Thus, for greater clarity of the following
argument one may assume from now that $w\neq\eps$.

\medskip
\noindent\emph{Sufficiency.}
Let $w$ be saturated. By Corollary~\ref{cor:star}, it is enough to show
that any $v\sub w$ with $\lz(v)=\lo(v)+1$ is such that inserting a digit
$1$ into an arbitrary slot of $v$ produces a subword of $w$. If $k=0$ then
no word $v$ with $\lz(v)=\lo(v)+1\ge1$ is
a subword of $w$. Let $k\ge1$ and let $v\sub w$ satisfy
$\lz(v)=m$, $\lo(v)=m-1$ for some $m\le k$.
Any word $v'$
obtained from $v$ by inserting a single digit $1$ satisfies $\lz(v')=m\le k$ and $\lo(v')=m\le k$, hence $v'\sub w$ by point (1) of Observation~\ref{greedy}. By
Corollary~\ref{cor:star}, $P_{w}(\N)=\frac12$.

\medskip
\noindent\emph{Necessity.}
Let $w$ be non-saturated. Then, by Corollary~\ref{cor:main-ii},
\[
P_{w}(\N)\;\ge\;\frac12+2^{-2\min(\lz(w),\,\lo(w)+1)-1}\;>\;\frac12 .
\]
\end{proof}


\begin{thebibliography}{9}

\bibitem{Besineau}
J.~B\'esineau,
\emph{Ind\'ependance statistique d'ensembles li\'es \`a la fonction ``somme
des chiffres''},
Acta Arith. \textbf{20} (1972), 401--416.

\bibitem{Cheng}
K.~Cheng,
\emph{A first-exit proof of Cusick's sum-of-digits conjecture},
arXiv:2606.23398, 2026.

\bibitem{DKS}
M.~Drmota, M.~Kauers and L.~Spiegelhofer,
\emph{On a conjecture of Cusick concerning the sum of digits of $n$ and
$n+t$},
SIAM J. Discrete Math. \textbf{30} (2016), no.~2, 621--649.

\bibitem{EmmeHubert}
J.~Emme and P.~Hubert,
\emph{Central limit theorem for probability measures defined by sum-of-digits
function in base 2},
Ann. Sc. Norm. Super. Pisa Cl. Sci. (5) \textbf{19} (2019), no.~2, 757--780.

\bibitem{MS}
J.~F. Morgenbesser and L.~Spiegelhofer,
\emph{A reverse order property of correlation measures of the sum-of-digits
function},
Integers \textbf{12} (2012), Paper No. A47, 5 pp.

\bibitem{SS}
B.~Sobolewski and L.~Spiegelhofer,
\emph{Decomposing the sum-of-digits correlation measure},
J. Number Theory \textbf{280} (2026), 702--736.

\bibitem{Spiegelhofer}
L.~Spiegelhofer,
\emph{A lower bound for Cusick's conjecture on the digits of $n+t$},
Math. Proc. Cambridge Philos. Soc. \textbf{172} (2022), no.~1, 139--161.

\bibitem{SW}
L.~Spiegelhofer and M.~Wallner,
\emph{The binary digits of $n+t$},
Ann. Sc. Norm. Super. Pisa Cl. Sci. (5) \textbf{24} (2023), no.~1, 1--31.

\bibitem{Tarlowski}
D.~Tar{\l}owski,
\emph{On the sum-of-digits measures and Cusick's conjecture via stopped
random walks},
arXiv:2605.08624, 2026.

\end{thebibliography}
\end{document}